\documentclass[fleqn, 11pt]{article}
\usepackage{amsmath, amssymb, amsthm, graphicx}
\usepackage{hyperref}
\theoremstyle{definition}
\newtheorem{theorem}{Theorem}[section]

\newtheorem{lemma}[theorem]{Lemma}
\newtheorem{corollary}[theorem]{Corollary}

\newtheorem{Remark}[theorem]{Remark}
\newenvironment{remark}{\begin{Remark}\rm}{\end{Remark}}

\newtheorem{Example}[theorem]{Example}
\newtheorem{Question}[theorem]{Question}

\numberwithin{equation}{section}
\title{Equal masses results for choreographies $n$-body problems}

\author{Pieter Tibboel \\
Department of Mathematical Sciences\\
Xi'an Jiaotong-Liverpool University\\
Suzhou, China\\
Pieter.Tibboel@xjtlu.edu.cn}

\begin{document}
\maketitle
\begin{abstract}
  We prove that equally spaced choreography solutions of a large class of $n$-body problems including the classical $n$-body problem and a subset of quasi-homogeneous $n$-body problems, have equal masses if the dimension of the space spanned by the point masses is $n-1$, $n-2$, or, if $n$ is odd, if the dimension is $n-3$. If $n$ is even and the dimension is $n-3$, then all masses with an odd label are equal and all masses with an even label are equal. Additionally, we prove that the same results hold true for any solution of an $n+1$-body problem for which $n$ of the point masses behave like an equally spaced choreography and the $n+1$st point mass is fixed at the origin. Furthermore, we deduce that if the curve along which the point masses of a choreography move has an axis of symmetry, the masses have to be equal if $n=3$ and that if $n=4$, if three of the point masses behave as stated and the fourth mass is fixed at a point, the masses of the first three point masses are all equal. Finally, we prove for the $n$-body problem in spaces of negative constant Gaussian curvature that if $n<6$, $n\neq 4$, equally spaced choreography solutions have to have equal masses, and for $n=4$ the even labeled masses are equal and the odd labeled masses are equal and that the same holds true for the $n$-body problem in spaces of positive constant Gaussian curvature, as long as the point masses do not move along a great circle. Additionally, we show that these last two results are also true for any solution to the $n+1$-body problem in spaces of negative constant Gaussian curvature and the $n+1$-body problem in spaces of positive constant Gaussian curvature respectively, for the case that $n$ of the point masses behave like an equally spaced choreography and the $n+1$st is fixed at a point.
\end{abstract}
\maketitle

\section{Introduction}
  In this paper we will consider two types of $n$-body problems, the first of which being the problem of finding the orbits of  point masses $q_{1}$,...,$q_{n}\in\mathbb{R}^{d}$, $d\in\mathbb{N}$, and respective  masses $m_{1}>0$,...,\textrm{ }$m_{n}>0$ as described by the system of differential equations
  \begin{align}\label{Equations of motion}
    \ddot{q}_{k}=\sum\limits_{\substack{j=1\\j\neq k}}^{n}m_{j}(q_{j}-q_{k})f\left(\|q_{j}-q_{k}\|^{2}\right),\textrm{ }k\in\{1,...,n\},
  \end{align}
  where $\|\cdot\|$ is the Euclidean norm, $f:\mathbb{R}_{>0}\rightarrow\mathbb{R}_{>0}$ is a positive-valued scalar function and $\sqrt{x}f(x)$ is a decreasing function. The study of $n$-body problems of this type has applications to, for example, atomic physics, celestial mechanics, chemistry, crystallography,
  differential equations and dynamical systems (see for example \cite{AbrahamMarsden}--\cite{D0}, \cite{DMS}, \cite{DPS4}, \cite{J}--\cite{PV} and the references therein).
  The second $n$-body problem we will investigate is the $n$-body problem in spaces of constant Gaussian curvature, or curved $n$-body problem for short, which can be formulated as follows:
    Let $\sigma=\pm 1$. The $n$-body problem in spaces of constant Gaussian curvature is the problem of finding the dynamics of point masses \begin{align*}q_{1},...,\textrm{ }q_{n}\in\mathbb{M}_{\sigma}^{2}=\{(x_{1},x_{2},x_{3})\in\mathbb{R}^{3}|x_{1}^{2}+x_{2}^{2}+\sigma x_{3}^{2}=\sigma\},\end{align*} with respective masses $m_{1}>0$,..., $m_{n}>0$, determined by the system of differential equations
  \begin{align}\label{EquationsOfMotion Curved}
   \ddot{q}_{k}=\sum\limits_{\substack{j=1\\j\neq k}}^{n}\frac{m_{j}(q_{j}-\sigma(q_{k}\odot q_{j})q_{k})}{(\sigma -\sigma(q_{k}\odot q_{j})^{2})^{\frac{3}{2}}}-\sigma(\dot{q}_{k}\odot\dot{q}_{k})q_{k},\textrm{ }k\in\{1,...,\textrm{ }n\},
  \end{align}
  where for $x$, $y\in\mathbb{M}_{\sigma}^{2}$  the product $\cdot\odot\cdot$ is defined as
  \begin{align*}
    x\odot y=x_{1}y_{1}+x_{2}y_{2}+\sigma x_{3}y_{3}.
  \end{align*}
  While the curved $n$-body problem for $n=2$ goes back as far as the 1830s, a working model for the $n\geq 2$ case was not found until 2008 by Diacu, P\'erez-Chavela and Santoprete (see \cite{DPS1}, \cite{DPS2} and \cite{DPS3}). This breakthrough then gave rise to further results for the $n\geq 2$ case in \cite{D1}--\cite{DK}, \cite{DPo}, \cite{DSS}, \cite{DT} and \cite{PS}--\cite{ZZ}. See \cite{BM}, \cite{BMK}, \cite{DK}, \cite{DPS1}, \cite{DPS2} and \cite{DPS3} for a historical overview. The study of the curved $n$-body problem has applications to for example geometric mechanics, Lie groups and algebras, non-Euclidean and differential geometry and stability theory, the theory of polytopes and topology (see for example \cite{D6}). A particular use of the curved $n$-body problem is that it may give information about the geometry of the universe. For example: Diacu, P\'erez-Chavela and Santoprete  showed that the configuration of the Sun, Jupiter and the Trojan asteroids cannot exist in curved space (see \cite{DPS1}, \cite{DPS2}). \newline

  By a choreography, or choreographic solution to (\ref{Equations of motion}), or (\ref{EquationsOfMotion Curved}) we mean any solution $q_{1},...,q_{n}$ for which there exists a twice-continuously differentiable periodic vector-valued function $p(t)$ and constants $h_{1},...,h_{n}$ such that $q_{k}(t)=p(t+h_{k})$, $k\in\{1,...,n\}$. For ease of notation, we define $h_{k+Kn}=h_{k}+K\widehat{P}$ for $k\in\{1,...,n\}$, where $\widehat{P}$ is the period of $p$, $K\in\mathbb{Z}$. If $h_{k+1}-h_{k}$ is independent of $k$, then we call a choreographic solution an \textit{equally spaced choreography}, or \textit{equally spaced choreographic solution}. Examples of choreographies are the well-known relative equilibrium where $n$ masses move along a circle, evenly distributed as if they were the vertices of a regular polygon (see for example \cite{AEP}, \cite{CHG}, \cite{T5} and \cite{T7} and the references therein), the famous figure eight, first discovered numerically by Moore (see \cite{Moore}) and independently discovered and proved by Chenciner and Montgomery (see \cite{CM}), which gave rise to the discovery of additional families of choreographies (see for example \cite{CGMS}, \cite{FFO}, \cite{S}, \cite{Mon}, \cite{Y} and the references therein) and numerically discovered choreographies for the curved $n$-body problem (see \cite{MG}, \cite{Mont}), of which the figure eight solution on $\mathbb{S}^{2}$ was analytically proven in \cite{SZ}. Chenciner proved in \cite{C} for general $n$ that for any equally spaced choreographic solution to the classical $n$-body problem, i.e.  $f(x)=x^{-\frac{3}{2}}$ (though it should be noted that his proof does not depend on the choice of $f$), with unequal masses, the same choreographic solution solves the $n$-body problem if all masses are taken equal. He then coined equally spaced choreographic solutions with unequal masses \textit{perverse choreographic solutions} and proved that if $n<6$, such solutions do not exist if the choreography lies in a plane. He additionally stated that choreographies also exist in three-space and that the problem of their perversity is completely open. For the case that a choreographic solution is not assumed to be equally spaced, nothing is known regarding whether the masses are equal even for $n=3$.
  \begin{remark}
    Throughout this paper, when we say that $d$ has a specific value $k\in\mathbb{N}$, we mean that there exists a set $S\subset\left[0,\widehat{P}\right]$ such that the point masses $q_{1}(t)$,...,$q_{n}(t)$ span a $k$-dimensional space for any $t\in S\subset\left[0,\widehat{P}\right]$, where $S$ has Lebesgue measure $\widehat{P}$.
  \end{remark}
  In this paper we will investigate the existence of perverse choreographies in higher-dimensional spaces, in spaces of constant Gaussian curvature and prove that if the curve along which the point masses of a choreography move has an axis of symmetry, the masses have to be equal for $n=3$. Specifically, we will prove the following results:
  \begin{theorem}\label{Main Theorem 1}
    Let $q_{1},...,q_{n}$ be an equally spaced choreographic solution of (\ref{Equations of motion}). If $d=n-1$, $d=n-2$, or if $n$ is odd and $d=n-3$, all masses are equal and if $n$ is even and $d=n-3$, all masses with an odd label are equal and all masses with an even label are equal.
  \end{theorem}
  \begin{corollary}\label{Main Corollary 1}
    Let $q_{1},...,q_{n+1}$ be a solution of (\ref{Equations of motion}), where $q_{1},...,q_{n}$ are an equally spaced choreography and $q_{n+1}=0$. If $d=n-1$, $d=n-2$, or $n$ is odd and $d=n-3$, then the first $n$ masses are equal and if $n$ is even and $d=n-3$, then the first $\frac{n}{2}$ masses with an odd label are equal and the first $\frac{n}{2}$ masses with an even label are equal.
  \end{corollary}
  \begin{theorem}\label{Main Theorem 2}
    Let $q_{1},...,q_{n}$ be a choreographic solution of (\ref{Equations of motion}). If $n=3$ and the curve along which the masses move has an axis of symmetry, then the masses $m_{1}$,$m_{2}$,$m_{3}$ are all equal.
  \end{theorem}
  \begin{corollary}\label{Main Corollary 2}
    Let $q_{1},...,q_{n+1}$ be a solution of (\ref{Equations of motion}), where $q_{1},...,q_{n}$ are a choreography and $q_{n+1}=0$. If $n=3$ and the curve along which the masses move has an axis of symmetry and $q_{n+1}=0$, then the masses $m_{1}$,$m_{2}$,$m_{3}$ are all equal.
  \end{corollary}
  \begin{theorem}\label{Main Theorem 3}
    Let $q_{1},...,q_{n}$ be an equally spaced choreographic solution of (\ref{EquationsOfMotion Curved}). If $\sigma=-1$, then if $n<6$, $n\neq 4$, all masses are equal and if $n=4$, then the even labeled masses are equal and the odd labeled masses are equal. If $\sigma=1$, the same holds true, unless the point masses move along a great circle.
  \end{theorem}
  \begin{corollary}\label{Main Corollary 3}
    Let $q_{1},...,q_{n+1}$ be a solution of (\ref{EquationsOfMotion Curved}), where $q_{1},...,q_{n}$ is an equally spaced choreography and $q_{n+1}=(0,0,1)^{T}$. If $\sigma=-1$, then if $n<6$, $n\neq 4$, the masses of the first $n$ point masses are equal and if $n=4$, then $m_{1}=m_{3}$ and $m_{2}=m_{4}$. If $\sigma=1$, the same holds true, unless the first $n$ point masses move along a great circle.
  \end{corollary}
  \begin{remark}
    Solutions of the type discussed in Corollary~\ref{Main Corollary 1}, Corollary~\ref{Main Corollary 2} and Corollary~\ref{Main Corollary 3} exist: It is well-known that one can construct a solution of (\ref{Equations of motion}) and (\ref{EquationsOfMotion Curved}) where the $q_{1}$,...,$q_{n}$ behave as the vertices of a regular polygon rotating around $q_{n+1}=0$ and $q_{n+1}=(0,0,1)^{T}$ respectively.
  \end{remark}
  \begin{remark}
    Choreographic solutions for which $d=n-k$, $k\in\mathbb{Z}_{\geq 1}$, $k\leq n-2$, exist. For example: We know that planar choreographic solutions exist for all $n$ (see \cite{CGMS}). If a choreographic solution is planar, then $d=2$, so choosing $n=k+2$ gives that for any $k$ there is an $n$ such that $d=n-k$. The matter of what choreographies exist for $d\geq 3$ is an open problem, but regardless of that, Theorem~\ref{Main Theorem 1} and Corollary~\ref{Main Corollary 1} are a first step in investigating the perversity of choreographies if $d\geq 3$.
  \end{remark}
  \begin{remark}
    The argument used in the proof of Theorem~\ref{Main Theorem 1} to prove that if $d=n-3$, $n$ odd, all masses are equal, or take on at most two values if $n$ is even, can be applied if $d=n-4$ and $n$ is even as well, (with $n-3$ replaced with $n-4$) under the condition that we can choose a $k\in\{1,...,n-1\}$ in such a way that the number of linearly independent vectors in the linear combinations of (\ref{General equations1}) and (\ref{General equations2}) is the same as in the linear combinations of (\ref{Lead up to general equations1}) and (\ref{Lead up to general equations2}), in which case if $n$ is even and $d=n-4$ all $a_{l}$ in (\ref{Span}), $l\neq\frac{n}{2}$ have to be zero. However, if $n>6$, it is possible that there is no such $k$ and excluding that possibility is nontrivial and likely requires further knowledge on the dynamics of the choreography.
  \end{remark}
  The remainder of this paper is as follows: We will first formulate necessary notation and lemmas in section~\ref{Section Background Theory} and then prove  Theorem~\ref{Main Theorem 1} in section~\ref{Section proof of main theorem 1}, Corollary~\ref{Main Corollary 1} in section~\ref{Section proof of main corollary 1}, Theorem~\ref{Main Theorem 2} in section~\ref{Section proof of main theorem 2}, Corollary~\ref{Main Corollary 2} in section~\ref{Section proof of main corollary 2}, Theorem~\ref{Main Theorem 3} in section~\ref{Section proof of main theorem 3} and Corollary~\ref{Main Corollary 3} in section~\ref{Section proof of main corollary 3}.
  \section{Background theory}\label{Section Background Theory}
  Throughout this paper we will use the notation introduced in the previous section for choreographies. For ease of notation, we will additionally define $m_{k+Kn}=m_{k}$ for $k\in\{1,...,n\}$, $K\in\mathbb{Z}$ and if we deal with equally spaced choreographies, we will choose $h_{k+1}-h_{k}=1$ and $\widehat{P}=n$, as we can always use a rescaling argument if a choreography is equally spaced and $h_{k+1}-h_{k}\neq 1$. Furthermore, we will assume that for any choreography solution of (\ref{Equations of motion}) we have that
      \begin{align}\label{Center of mass zero}
        \sum\limits_{k=1}^{n}m_{k}q_{k}=0.
      \end{align}
      Note that by (\ref{Equations of motion}) we have that
      \begin{align*}
        \sum\limits_{k=1}^{n}m_{k}\ddot{q}_{k}&=\sum\limits_{k=1}^{n}\sum\limits_{\substack{j=1\\j\neq k}}^{n}m_{k}m_{j}(q_{j}-q_{k})f\left(\|q_{j}-q_{k}\|^{2}\right)=0,
      \end{align*}
      so there exist constant vectors $A$ and $B$ such that $\sum\limits_{k=1}^{n}m_{k}q_{k}=At+B$. But for the $q_{k}$, $k\in\{1,...,n\}$, to lie on a closed, periodic curve, we need that $A=0$ and if $B\neq 0$, we can replace the $q_{k}$, $k\in\{1,...,n\}$ in (\ref{Equations of motion}) with $\widehat{q}_{k}=q_{k}-\frac{1}{\sum\limits_{j=1}^{n}m_{j}}B$ and work with $\widehat{q}_{k}$ instead. This result is well-known, but was quickly proven to make the paper as self-contained as possible. For ease of notation, we will define $\Delta_{j}(t)=p(t+j)-p(t)$ for all $j\in\{1,...,n-1\}$ and $M=\sum\limits_{j=1}^{n}m_{j}$. Additionally, to prove Theorem~\ref{Main Theorem 1} and Theorem~\ref{Main Theorem 3} the following well-known vectors will be helpful: Let $\widehat{e}_{1}$,..., $\widehat{e}_{n}\in\mathbb{C}^{n}$, where the $j$th component of $\widehat{e}_{l}$ is $\frac{1}{\sqrt{n}}e^{\frac{2\pi li}{n}(j-1)}$, $j\in\{1,...,n\}$, $l\in\mathbb{Z}$. Let $\lambda_{l}=e^{\frac{2\pi li}{n}}$. Note that if $\widehat{B}$ is the $n\times n$ matrix for which the $j$th component of $\widehat{B}v$ is the $j+1$st component of $v$ for all vectors $v\in\mathbb{C}^{n}$, then $\widehat{B}\widehat{e}_{l}=\lambda_{l}\widehat{e}_{l}$. Additionally, note that the Euclidean inner product of $\widehat{e}_{j}$ and $\widehat{e}_{l}$ is zero for $j\neq l$ and $1$ for $j=l$, making the vectors $\widehat{e}_{1}$,..., $\widehat{e}_{n}$ an orthonormal basis of $\mathbb{C}^{n}$. Finally, if the curve with parametrisation $p(t)$ lies in the plane and is axisymmetric, we will write
      \begin{align}\label{reflection identity}
        p(-t)=\begin{pmatrix}
        1 & 0 \\0 & -1
      \end{pmatrix}p(t).
      \end{align}
      Next we will formulate the following lemmas needed to prove Theorem~\ref{Main Theorem 1}, Corollary~\ref{Main Corollary 1}, Theorem~\ref{Main Theorem 2}, Corollary~\ref{Main Corollary 2}, Theorem~\ref{Main Theorem 3} and Corollary~\ref{Main Corollary 3}:
      \begin{lemma}\label{Lemma 1}
        Let $q_{1}$,...,$q_{n}$ be an equally spaced choreography of (\ref{Equations of motion}). Then
        \begin{align}\label{Linear combination f}
          0=\sum\limits_{j=1}^{n-1}\left(m_{k+j}-\frac{M}{n}\right)\Delta_{j}(t)f\left(\|\Delta_{j}(t)\|^{2}\right)
        \end{align}
        and
        \begin{align}\label{Linear combination c}
          0=\sum\limits_{j=1}^{n-1}\left(m_{k+j}-\frac{M}{n}\right)\Delta_{j}(t),\textrm{ }k\in\{1,...,n\}.
        \end{align}
      \end{lemma}
      \begin{proof}
        Substituting $q_{k}(t)$ with $p(t+k)$ and $q_{j}(t)$ with $p(t+j)$ in (\ref{Equations of motion}) and subsequently replacing $t+k$ with $t$ gives
        \begin{align*}
          \ddot{p}(t)=\sum\limits_{\substack{j=1\\j\neq k}}^{n}m_{j}(p(t+j-k)-p(t))f\left(\|p(t+j-k)-p(t)\|^{2}\right),
        \end{align*}
        which can be rewritten as
        \begin{align}\label{Sum 1f}
          \ddot{p}(t)=\sum\limits_{j=1}^{n-1}m_{j+k}(p(t+j)-p(t))f\left(\|p(t+j)-p(t)\|^{2}\right)=\sum\limits_{j=1}^{n-1}m_{j+k}\Delta_{j}(t)f\left(\|\Delta_{j}(t)\|^{2}\right)
        \end{align}
        for any fixed value $k$. Summing both sides of (\ref{Sum 1f}) from $1$ to $n$ with respect to $k$ then gives
        \begin{align}\label{Sum 2f}
          n\ddot{p}(t)&=\sum\limits_{j=1}^{n-1}\left(\sum\limits_{k=1}^{n}m_{j+k}\right)\Delta_{j}(t)f\left(\|\Delta_{j}(t)\|^{2}\right)=\sum\limits_{j=1}^{n-1}M\Delta_{j}(t)f\left(\|\Delta_{j}(t)\|^{2}\right).
        \end{align}
        Dividing both sides of (\ref{Sum 2f}) by $n$ and subtracting the resulting equation from (\ref{Sum 1f}) then gives (\ref{Linear combination f}).
        Next, we will prove (\ref{Linear combination c}) by using that $\sum\limits_{k=1}^{n}m_{k}q_{k}=0$: Note that
        \begin{align*}
          Mq_{k}=\sum\limits_{j=1}^{n}m_{j}q_{k}-0=\sum\limits_{j=1}^{n}m_{j}q_{k}-\sum\limits_{j=1}^{n}m_{j}q_{j}=\sum\limits_{\substack{j=1\\j\neq k}}^{n}m_{j}(q_{k}-q_{j}),
        \end{align*}
        which, writing $q_{k}(t)=p(t+k)$, $q_{j}(t)=p(t+j)$ and replacing $t+k$ with $t$, can be rewritten as
        \begin{align*}
          -Mp(t)=\sum\limits_{\substack{j=1\\j\neq k}}^{n}m_{j}(p(t+j-k)-p(t)),
        \end{align*}
        which, for any fixed $k$, can be rewritten as
        \begin{align}\label{Sum 2c}
          -Mp(t)=\sum\limits_{j=1}^{n-1}m_{j+k}(p(t+j)-p(t))=\sum\limits_{j=1}^{n-1}m_{j+k}\Delta_{j}(t).
        \end{align}
        Summing both sides of (\ref{Sum 2c}) with respect to $k$ from $1$ to $n$ and dividing the resulting equation on both sides by $n$ then gives
        \begin{align}\label{Sum 3c}
          -Mp(t)=\sum\limits_{j=1}^{n-1}\frac{M}{n}\Delta_{j}(t).
        \end{align}
        Subtracting (\ref{Sum 3c}) from (\ref{Sum 2c}) then finally gives (\ref{Linear combination c}). This completes the proof.
      \end{proof}
      \begin{remark}
        Lemma~\ref{Lemma 1} was proven by Chenciner in \cite{C} for $f(x)=x^{-\frac{3}{2}}$ and his proof works for general $f$ as well. As technically speaking the proof in \cite{C} was written down only for $f(x)=x^{-\frac{3}{2}}$ and to make the paper self-contained, we have included the proof here.
      \end{remark}
  \begin{lemma}\label{Lemma 2}
    Consider any choreography of (\ref{Equations of motion}) for which the curve given by $p(t)$ has an axis of symmetry. Then for all $l$, $k\in\{1,...,n\}$, $l\neq k$ we have that
    \begin{align}\label{Equations of motion choreography result 1}
      &(m_{k}-m_{l})\bigg(p(t+h_{l}-h_{k})-p(t)\bigg)f\left(\|p(t+h_{l}-h_{k})-p(t)\|^{2}\right)\nonumber\\
      &=\sum\limits_{\substack{j=1\\j\neq k,l}}^{n}m_{j}\Bigg(\bigg(p(t+h_{j}-h_{k})-p(t)\bigg)f\left(\|p(t+h_{j}-h_{k})-p(t)\|^{2}\right)\nonumber\\
      &-\bigg(p(t-(h_{j}-h_{l}))-p(t)\bigg)f\left(\|p(t-(h_{j}-h_{l}))-p(t)\|^{2}\right)\Bigg),
    \end{align}
    we have
    \begin{align}\label{Equations of motion choreography result 2}
      &(m_{k}-m_{l})\bigg(p(t+(h_{l}-h_{k}))-p(t)\bigg)f\left(\|p(t+(h_{l}-h_{k}))-p(t)\|^{2}\right)\nonumber\\
      &=\sum\limits_{\substack{j=1\\j\neq k,l}}^{n}m_{j}\Bigg(\bigg(p(t+(h_{j}-h_{k}))-p(t+(h_{l}-h_{k}))\bigg)f\left(\|p(t+(h_{j}-h_{k}))-p(t+(h_{l}-h_{k}))\|^{2}\right)\nonumber\\
      &-\bigg(p(t+(h_{l}-h_{j}))-p(t+(h_{l}-h_{k}))\bigg)f\left(\|p(t+(h_{l}-h_{j}))-p(t+(h_{l}-h_{k}))\|^{2}\right)\Bigg),
    \end{align}
    we have
    \begin{align}\label{Equations of motion choreography result 3}
      &(m_{k}-m_{l})\bigg(p(t+h_{l}-h_{k})-p(t)\bigg)\nonumber\\
      &=\sum\limits_{\substack{j=1\\j\neq k,l}}^{n}m_{j}\Bigg(\bigg(p(t+h_{j}-h_{k})-p(t)\bigg)-\bigg(p(t-(h_{j}-h_{l}))-p(t)\bigg)\Bigg),
    \end{align}
    and we have
    \begin{align}\label{Equations of motion choreography result 4}
      &(m_{k}-m_{l})\bigg(p(t+(h_{l}-h_{k}))-p(t)\bigg)\nonumber\\
      &=\sum\limits_{\substack{j=1\\j\neq k,l}}^{n}m_{j}\Bigg(\bigg(p(t+(h_{j}-h_{k}))-p(t+(h_{l}-h_{k}))\bigg)-\bigg(p(t+(h_{l}-h_{j}))-p(t+(h_{l}-h_{k}))\bigg)\Bigg).
    \end{align}
  \end{lemma}
  \begin{proof}
  Substituting the $q_{k}$ and $q_{j}$ with $p(t+h_{k})$ and $p(t+h_{j})$ respectively in (\ref{Equations of motion}) we get that
  \begin{align*}
    \ddot{p}(t+h_{k})=\sum\limits_{\substack{j=1\\j\neq k}}^{n}m_{j}\bigg(p(t+h_{j})-p(t+h_{k})\bigg)f\left(\|p(t+h_{j})-p(t+h_{k})\|^{2}\right),
  \end{align*}
  which holds for all $t\in\mathbb{R}$. As such we may replace $t+h_{k}$ with $t$ to obtain
  \begin{align}\label{Equations of motion choreography 1}
    \ddot{p}(t)=\sum\limits_{\substack{j=1\\j\neq k}}^{n}m_{j}\bigg(p(t+h_{j}-h_{k})-p(t)\bigg)f\left(\|p(t+h_{j}-h_{k})-p(t)\|^{2}\right).
  \end{align}
  Let $s=-t$. Substituting $t$ with $-s$ in (\ref{Equations of motion choreography 1}) and using (\ref{reflection identity}) we get
  \begin{align}\label{Equations of motion choreography 2}
    &\begin{pmatrix}
      1 & 0 \\
      0 & -1
    \end{pmatrix}(-1)^{2}\frac{d^{2}}{ds^{2}}{p}(s)\nonumber\\
    &=\begin{pmatrix}
      1 & 0 \\
      0 & -1
    \end{pmatrix}\sum\limits_{\substack{j=1\\j\neq k}}^{n}m_{j}\bigg(p(s-h_{j}+h_{k})-p(s)\bigg)f\left(\|p(s-h_{j}+h_{k})-p(s)\|^{2}\right)
  \end{align}
  and multiplying both sides of (\ref{Equations of motion choreography 2}) from the left with $\begin{pmatrix}
      1 & 0 \\
      0 & -1
    \end{pmatrix}^{-1}$ gives
  \begin{align}\label{Equations of motion choreography 3}
    \ddot{p}(s)=\sum\limits_{\substack{j=1\\j\neq k}}^{n}m_{j}\bigg(p(s-h_{j}+h_{k})-p(s)\bigg)f\left(\|p(s-h_{j}+h_{k})-p(s)\|^{2}\right).
  \end{align}
  Replacing $s$ with $t$ and $k$ with $l$, $l\in\{1,...,n\}$ in (\ref{Equations of motion choreography 3}) and subtracting the resulting equation from (\ref{Equations of motion choreography 1}) gives
  \begin{align}\label{Equations of motion choreography 4}
    &0=(m_{l}-m_{k})\bigg(p(t+h_{l}-h_{k})-p(t)\bigg)f\left(\|p(t+h_{l}-h_{k})-p(t)\|^{2}\right)\nonumber\\
    &+\sum\limits_{\substack{j=1\\j\neq k,l}}^{n}m_{j}\Bigg(\bigg(p(t+h_{j}-h_{k})-p(t)\bigg)f\left(\|p(t+h_{j}-h_{k})-p(t)\|^{2}\right)\nonumber\\
    &-\bigg(p(t-(h_{j}-h_{l}))-p(t)\bigg)f\left(\|p(t-(h_{j}-h_{l}))-p(t)\|^{2}\right)\Bigg),
  \end{align}
  which proves (\ref{Equations of motion choreography result 1}).
  Replacing $t$ with $-s-(h_{l}-h_{k})$ in (\ref{Equations of motion choreography 4}), multiplying both sides of the resulting equation from the left with $-\begin{pmatrix}
    1 & 0 \\0 & -1
  \end{pmatrix}^{-1}$ and using (\ref{reflection identity}) gives
  \begin{align}\label{Equations of motion choreography 5}
    &0=(m_{l}-m_{k})\bigg(p(s+(h_{l}-h_{k}))-p(s)\bigg)f\left(\|p(s+(h_{l}-h_{k}))-p(s)\|^{2}\right)\nonumber\\
    &+\sum\limits_{\substack{j=1\\j\neq k,l}}^{n}m_{j}\Bigg(\bigg(p(s+(h_{j}-h_{k}))-p(s+(h_{l}-h_{k}))\bigg)f\left(\|p(s+(h_{j}-h_{k}))-p(s+(h_{l}-h_{k}))\|^{2}\right)\nonumber\\
    &-\bigg(p(s+(h_{l}-h_{j}))-p(s+(h_{l}-h_{k}))\bigg)f\left(\|p(s+(h_{l}-h_{j}))-p(s+(h_{l}-h_{k}))\|^{2}\right)\Bigg).
  \end{align}
  Replacing $s$ with $t$ in (\ref{Equations of motion choreography 5}) then proves (\ref{Equations of motion choreography result 2}). \\
  Note that by (\ref{Center of mass zero}) we have that
  \begin{align}\label{Center of mass zero 1}
    -Mq_{k}=\sum\limits_{j=1}^{n}m_{j}(-q_{k})+0=-\sum\limits_{j=1}^{n}m_{j}q_{k}+\sum\limits_{j=1}^{n}m_{j}q_{j}=\sum\limits_{j=1}^{n}m_{j}(q_{j}-q_{k})=\sum\limits_{\substack{j=1\\j\neq k}}^{n}m_{j}(q_{j}-q_{k})
  \end{align}
  and if we replace $q_{k}$ with $p(t+h_{k})$ for all $k\in\{1,...,n\}$ in (\ref{Center of mass zero 1}) and replace $t+h_{k}$ with $t$ in the resulting equation, then we get
  \begin{align}\label{Equations of motion choreography 1 center of mass}
    -Mp(t)=\sum\limits_{\substack{j=1\\j\neq k}}^{n}m_{j}\bigg(p(t+h_{j}-h_{k})-p(t)\bigg).
  \end{align}
  Note that the right-hand side of (\ref{Equations of motion choreography 1 center of mass}) is exactly (\ref{Equations of motion choreography 1}) with $f$ replaced by $1$ and $\ddot{p}(t)$ replaced by $-Mp(t)$. Repeating the steps that led from (\ref{Equations of motion choreography 1}) to (\ref{Equations of motion choreography 4}) and (\ref{Equations of motion choreography 5}) but starting with (\ref{Equations of motion choreography 1 center of mass}) instead then proves (\ref{Equations of motion choreography result 3}) and (\ref{Equations of motion choreography result 4}). This completes the proof.
  \end{proof}
  \begin{lemma}\label{Lemma 3}
    Let $q_{1}$,..., $q_{n}$ be an equally spaced choreographic solution of (\ref{EquationsOfMotion Curved}). Then
    \begin{align}\label{CurvedHelp}
       0=\sum\limits_{j=1}^{n-1}\frac{\left(m_{j+k}-\frac{M}{n}\right)\Bigg(p(t+j)-\sigma\bigg(p(t+j)\odot p(t)\bigg)p(t)\Bigg)}{\Bigg(\sigma -\sigma\bigg(p(t+j)\odot p(t)\bigg)^{2}\Bigg)^{\frac{3}{2}}}
    \end{align}
  for all $k\in\{1,...,n\}$.
  \end{lemma}
  \begin{proof}
    Let $q_{1}$,..., $q_{n}$ be an equally spaced choreographic solution of (\ref{EquationsOfMotion Curved}). Then by (\ref{EquationsOfMotion Curved}) we have that
    \begin{align}\label{Step1 Curved}
   \ddot{p}(t+k)&=\sum\limits_{j=1,\textrm{ }j\neq k}^{n}\frac{m_{j}\Bigg(p(t+j)-\sigma\bigg(p(t+j)\odot p(t+k)\bigg)p(t+k)\Bigg)}{\Bigg(\sigma -\sigma\bigg(p(t+j)\odot p(t+k)\bigg)^{2}\Bigg)^{\frac{3}{2}}}\nonumber\\&-\sigma\bigg(\dot{p}(t+k)\odot\dot{p}(t+k)\bigg)p(t+k),\textrm{ }k\in\{1,...,\textrm{ }n\}.
  \end{align}
  Let $s=t+k$. Then we can rewrite (\ref{Step1 Curved}) as
  \begin{align}\label{Step2 Curved}
   \ddot{p}(s)&=\sum\limits_{j=1,\textrm{ }j\neq k}^{n}\frac{m_{j}\Bigg(p(s+j-k)-\sigma\bigg(p(s+j-k)\odot p(s)\bigg)p(s)\Bigg)}{\Bigg(\sigma -\sigma\bigg(p(s+j-k)\odot p(s)\bigg)^{2}\Bigg)^{\frac{3}{2}}}-\sigma\bigg(\dot{p}(s)\odot\dot{p}(s)\bigg)p(s)\nonumber\\
   &=\sum\limits_{j=1}^{n-1}\frac{m_{j+k}\Bigg(p(s+j)-\sigma\bigg(p(s+j)\odot p(s)\bigg)p(s)\Bigg)}{\Bigg(\sigma -\sigma\bigg(p(s+j)\odot p(s)\bigg)^{2}\Bigg)^{\frac{3}{2}}}-\sigma\bigg(\dot{p}(s)\odot\dot{p}(s)\bigg)p(s).
  \end{align}
  Summing both sides of (\ref{Step2 Curved}) from $1$ to $n$ with respect to $k$ and subsequently dividing both sides by $n$ gives
  \begin{align}\label{Step3 Curved}
   \ddot{p}(s)&=\sum\limits_{j=1}^{n-1}\frac{\frac{M}{n}\Bigg(p(s+j)-\sigma\bigg(p(s+j)\odot p(s)\bigg)p(s)\Bigg)}{\Bigg(\sigma -\sigma\bigg(p(s+j)\odot p(s)\bigg)^{2}\Bigg)^{\frac{3}{2}}}-\sigma\bigg(\dot{p}(s)\odot\dot{p}(s)\bigg)p(s).
  \end{align}
  Subtracting (\ref{Step3 Curved}) from (\ref{Step2 Curved}) and replacing $s$ with $t$ then finally gives
  \begin{align*}
       0=\sum\limits_{j=1}^{n-1}\frac{\left(m_{j+k}-\frac{M}{n}\right)\Bigg(p(t+j)-\sigma\bigg(p(t+j)\odot p(t)\bigg)p(t)\Bigg)}{\Bigg(\sigma -\sigma\bigg(p(t+j)\odot p(t)\bigg)^{2}\Bigg)^{\frac{3}{2}}}
    \end{align*}
  for all $k\in\{1,...,n\}$. This completes the proof.
  \end{proof}
  \section{Proof of Theorem~\ref{Main Theorem 1}}\label{Section proof of main theorem 1}
    By Lemma~\ref{Lemma 1}, we have that
    \begin{align}\label{number 1}
          0=\sum\limits_{j=1}^{n-1}\left(m_{k+j}-\frac{M}{n}\right)\Delta_{j}(t)f\left(\|\Delta_{j}(t)\|^{2}\right)
    \end{align}
    and
    \begin{align}\label{number 2}
          0=\sum\limits_{j=1}^{n-1}\left(m_{k+j}-\frac{M}{n}\right)\Delta_{j}(t),\textrm{ }k\in\mathbb{Z}.
    \end{align}
    If $d=n-1$, then there exists a subset $S$ of $[0,n]$ of Lebesgue measure $n$ for which the vectors $\Delta_{1}(t),...,\Delta_{n-1}(t)$ are linearly independent for all $t\in S$, which means that by both (\ref{number 1}) and (\ref{number 2}) we have that $m_{j+k}=\frac{M}{n}$ for all $k$, $j\in\mathbb{Z}$, which proves that if $d=n-1$ all masses are equal. If $d=n-2$, there exists an $S\subset[0,n]$ of Lebesgue measure $n$, such that for all $t\in S$ we have that the vectors $\Delta_{1}(t),...,\Delta_{n-1}(t)$ span an $n-2$-dimensional space. For any $t\in S$ there then exists an $r\in\{1,...,n-1\}$ (which may depend on $t$), such that by (\ref{number 1}) and (\ref{number 2}) we have that
    \begin{align}\label{number 3}
          -\left(m_{k+r}-\frac{M}{n}\right)\Delta_{r}(t)f\left(\|\Delta_{r}(t)\|^{2}\right)=\sum\limits_{\substack{j=1\\j\neq r}}^{n-1}\left(m_{k+j}-\frac{M}{n}\right)\Delta_{j}(t)f\left(\|\Delta_{j}(t)\|^{2}\right)
    \end{align}
    and
    \begin{align}\label{number 4}
          -\left(m_{k+r}-\frac{M}{n}\right)\Delta_{r}(t)=\sum\limits_{\substack{j=1\\j\neq r}}^{n-1}\left(m_{k+j}-\frac{M}{n}\right)\Delta_{j}(t),
    \end{align}
    where the vectors $\Delta_{j}(t)$, $j\in\{1,...,n-1\}$, $j\neq r$, span an $n-2$-dimensional space. The $\Delta_{j}(t)$ are linearly independent for $j\in\{1,...,n-1\}$, $j\neq r$, so for those $j$ we have that if we multiply both sides of (\ref{number 4}) with $f\left(\|\Delta_{r}(t)\|^{2}\right)$ and compare the coefficients of the $\Delta_{j}(t)$ of the resulting identity with the coefficients of the $\Delta_{j}(t)$ in (\ref{number 3}), we find that  \begin{align*}\left(m_{k+j}-\frac{M}{n}\right)f\left(\|\Delta_{r}(t)\|^{2}\right)=\left(m_{k+j}-\frac{M}{n}\right)f\left(\|\Delta_{j}(t)\|^{2}\right)\end{align*} for all $k\in\mathbb{Z}$, $j\in\{1,...,n-1\}$. So if there is at least one $t\in S$ for which there are an $r$ and a $j$ for which  $f\left(\|\Delta_{j}(t)\|^{2}\right)\neq f\left(\|\Delta_{r}(t)\|^{2}\right)$, then all masses are equal. The alternative is that $f\left(\|\Delta_{j_{1}}(t)\|^{2}\right)=f\left(\|\Delta_{j_{2}}(t)\|^{2}\right)$ for all $j_{1}$, $j_{2}\in\{1,...,n-1\}$ and all $t\in S$. But if that is the case, then because $f$ is a strictly decreasing and therefore bijective function, we have that $\|\Delta_{j_{1}}(t)\|=\|\Delta_{j_{2}}(t)\|$ for all $j_{1}$, $j_{2}\in\{1,...,n-1\}$, for all $t\in[0,n]$, which means that, by using a suitable change of variables, we have that $\|q_{j_{1}}(t)-q_{j_{2}}(t)\|=\|q_{l_{1}}(t)-q_{l_{2}}(t)\|$ for all $j_{1},j_{2},l_{1},l_{2}\in\{1,...,n\}$, $j_{1}\neq j_{2}$ and $l_{1}\neq l_{2}$, for all $t\in [0,n]$, which means that the $q_{1}$,...,$q_{n}$ represent the vertices of an $n-1$-dimensional simplex, which contradicts that $d=n-2$. So all masses are equal.

    If $d=n-r$, $r\geq 3$, we have that there exists a set $S\subset[0,n]$ with Lebesgue measure $n$, for which the vectors $\Delta_{j}(t)$, $j\in\{1,...,n-1\}$ span an $n-r$-dimensional space for all $t\in S$. Additionally, note that because $\widehat{e}_{1}$,..., $\widehat{e}_{n}$ span $\mathbb{C}^{n}$, there exist constants $a_{1},...,a_{n-1}$ such that
    \begin{align}\label{Span}
      \begin{pmatrix}
        m_{1+k}-\frac{M}{n}\\
        m_{2+k}-\frac{M}{n}\\
        \vdots\\
        m_{n+k}-\frac{M}{n}
      \end{pmatrix}=\sum\limits_{l=1}^{n-1}a_{l}\lambda_{l}^{k}\widehat{e}_{l}.
    \end{align}
    Note that by construction $\widehat{e}_{n}$ is orthogonal to $\left(m_{1+k}-\frac{M}{n},m_{2+k}-\frac{M}{n},...,m_{n+k}-\frac{M}{n}\right)^{T}$, so we may exclude $\widehat{e}_{n}$ in the linear combination in (\ref{Span}).
    Again by Lemma~\ref{Lemma 1}, this means that
    \begin{align*}
          0=\sum\limits_{j=1}^{n-1}\left(m_{k+j}-\frac{M}{n}\right)\Delta_{j}(t)f\left(\|\Delta_{j}(t)\|^{2}\right)=\frac{1}{\sqrt{n}}\sum\limits_{j=1}^{n-1}\left(\sum\limits_{l=1}^{n-1}a_{l}\lambda_{l}^{k+j-1}\right)\Delta_{j}(t)f\left(\|\Delta_{j}(t)\|^{2}\right)
    \end{align*}
    and
    \begin{align*}
          0=\sum\limits_{j=1}^{n-1}\left(m_{k+j}-\frac{M}{n}\right)\Delta_{j}(t)=\frac{1}{\sqrt{n}}\sum\limits_{j=1}^{n-1}\left(\sum\limits_{l=1}^{n-1}a_{l}\lambda_{l}^{k+j-1}\right)\Delta_{j}(t),
    \end{align*}
    so
    \begin{align*}
          0=\sum\limits_{l=1}^{n-1}a_{l}\lambda_{l}^{k}\sum\limits_{j=1}^{n-1}\lambda_{l}^{j-1}\Delta_{j}(t)f\left(\|\Delta_{j}(t)\|^{2}\right)
    \textrm{ and }
          0=\sum\limits_{l=1}^{n-1}a_{l}\lambda_{l}^{k}\sum\limits_{j=1}^{n-1}\lambda_{l}^{j-1}\Delta_{j}(t),
    \end{align*}
    which means by the linear independence of the $\lambda_{l}^{k}$ as functions of $k$ that for all $l\in\{1,...,n-1\}$ we have that $a_{l}=0$, or
    \begin{align}
          0=\sum\limits_{j=1}^{n-1}\lambda_{l}^{j-1}\Delta_{j}(t)f\left(\|\Delta_{j}(t)\|^{2}\right)\textrm{ and }\label{DeltaForce1}\\
          0=\sum\limits_{j=1}^{n-1}\lambda_{l}^{j-1}\Delta_{j}(t).\label{DeltaForce2}
    \end{align}
    If there are any $l$ for which $a_{l}\neq 0$, $l\neq\frac{n}{2}\textrm{ }mod\textrm{ }n$ if $n$ is even, then taking complex conjugates on both sides of (\ref{DeltaForce1}) and (\ref{DeltaForce2}) gives
    \begin{align}
          0=\sum\limits_{j=1}^{n-1}\lambda_{\pm l}^{j-1}\Delta_{j}(t)f\left(\|\Delta_{j}(t)\|^{2}\right)\textrm{ and }\label{Lead up to general equations1}\\
          0=\sum\limits_{j=1}^{n-1}\lambda_{\pm l}^{j-1}\Delta_{j}(t),\label{Lead up to general equations2}
    \end{align}
    so for all $k\in\{1,...,n\}$ we have that
    \begin{align*}
          &0=\lambda_{l}^{k-1}\sum\limits_{j=1}^{n-1}\lambda_{-l}^{j-1}\Delta_{j}(t)f\left(\|\Delta_{j}(t)\|^{2}\right)-\lambda_{-l}^{k-1}\sum\limits_{j=1}^{n-1}\lambda_{l}^{j-1}\Delta_{j}(t)f\left(\|\Delta_{j}(t)\|^{2}\right)\textrm{ and }\\
          &0=\lambda_{l}^{k-1}\sum\limits_{j=1}^{n-1}\lambda_{-l}^{j-1}\Delta_{j}(t)-\lambda_{-l}^{k-1}\sum\limits_{j=1}^{n-1}\lambda_{l}^{j-1}\Delta_{j}(t),
    \end{align*}
    so as $\lambda_{l}^{k-1}\lambda_{-l}^{j-1}-\lambda_{-l}^{k-1}\lambda_{-l}^{k}\lambda_{l}^{j-1}=2i\sin{\frac{2\pi l}{n}(k-j)}$ that means that
    \begin{align*}
          0=\sum\limits_{\substack{j=1\\j\neq k}}^{n-1}\left(\sin{\frac{2\pi l}{n}(k-j)}\right)\Delta_{j}(t)f\left(\|\Delta_{j}(t)\|^{2}\right)\textrm{ and }\\
          0=\sum\limits_{\substack{j=1\\j\neq k}}^{n-1}\left(\sin{\frac{2\pi l}{n}(k-j)}\right)\Delta_{j}(t)
    \end{align*}
    for all $k\in\{1,...,n\}$. Particularly, if $n$ is even, this means that
    \begin{align}
          0=\sum\limits_{\substack{j=1\\j\neq k,k+\frac{n}{2}}}^{n-1}\left(\sin{\frac{2\pi l}{n}(k-j)}\right)\Delta_{j}(t)f\left(\|\Delta_{j}(t)\|^{2}\right)\textrm{ and }\label{General equations1}\\
          0=\sum\limits_{\substack{j=1\\j\neq k,k+\frac{n}{2}}}^{n-1}\left(\sin{\frac{2\pi l}{n}(k-j)}\right)\Delta_{j}(t)\label{General equations2}
    \end{align}
    for all $k\in\{1,...,n\}$. \\
    So if $d=n-3$, we have, reusing the argument we used for the $d=n-2$ case, that there are two ways we can write a vector as a linear combination of $n-3$ linearly independent vectors, unless $\|\Delta_{j}(t)\|$ is independent of $j$, in which case the point masses are the vertices of an $n$-simplex, which is a contradiction.   This means that if $d=n-1$, $d=n-2$, or if $n$ is odd and $d=n-3$, all masses are equal (as the $a_{l}$ are all zero) and if $n$ is even and $d=n-3$, all masses with an odd label are equal and all masses with an even label are equal (as we have not excluded the possibility that $a_{\frac{n}{2}}\neq 0$). This completes the proof.
  \section{Proof of Corollary~\ref{Main Corollary 1}}\label{Section proof of main corollary 1}
  If $q_{1}$,..., $q_{n}$ move along a curve like an equally spaced choreography and have masses $m_{1},...,m_{n}$ and $q_{n+1}=0$ and has mass $m$, then by (\ref{Equations of motion}) we have for $k\neq n+1$ that
  \begin{align*}
    \ddot{q}_{k}=\sum\limits_{\substack{j=1\\j\neq k}}^{n}m_{j}(q_{j}-q_{k})f\left(\|q_{j}-q_{k}\|^{2}\right)-m q_{k}f\left(\|q_{k}\|^{2}\right),
  \end{align*}
  which means that the same argument that gave (\ref{Sum 1f}) gives
  \begin{align}\label{the n+1 case 1}
     \ddot{p}(t)+m p(t)f\left(\|p(t)\|^{2}\right)&=\sum\limits_{j=1}^{n-1}m_{j+k}(p(t+j)-p(t))f\left(\|p(t+j)-p(t)\|^{2}\right)\nonumber\\
     &=\sum\limits_{j=1}^{n-1}m_{j+k}\Delta_{j}(t)f\left(\|\Delta_{j}(t)\|^{2}\right),
  \end{align}
  if we use our definition of $m_{j+k}$, $j$, $k\in\mathbb{Z}$ for the masses $m_{1}$,...,$m_{n}$.

  Summing both sides of (\ref{the n+1 case 1}) with respect to $k$ from $1$ to $n$ and dividing both sides of the resulting equation by $n$ then gives
  \begin{align}\label{the n+1 case 2}
     \ddot{p}(t)+m p(t)f\left(\|p(t)\|^{2}\right)=\sum\limits_{j=1}^{n-1}\frac{M}{n}\Delta_{j}(t)f\left(\|\Delta_{j}(t)\|^{2}\right)
  \end{align}
  and subtracting (\ref{the n+1 case 2}) from (\ref{the n+1 case 1}) then gives (\ref{Linear combination f}) again. As the formula for the center of mass is exactly the same as in the proof of Lemma~\ref{Lemma 1}, (\ref{Linear combination c}) holds true as well and as (\ref{Linear combination f}) and (\ref{Linear combination c}) generated the proof of Theorem~\ref{Main Theorem 1}, the proof of Theorem~\ref{Main Theorem 1} proves Corollary~\ref{Main Corollary 1} as well.
  \section{Proof of Theorem~\ref{Main Theorem 2}}\label{Section proof of main theorem 2}
  Consider a choreography solution of (\ref{Equations of motion}) for $n=3$ for which $p$ has an axis of symmetry. Then by Lemma~\ref{Lemma 2} we have that for any $i$, $j$, $k\in\{1,2,3\}$, $i\neq j$, $i\neq k$, $j\neq k$,
  \begin{align}\label{Boom1}
      &(m_{i}-m_{k})\bigg(p(t+h_{k}-h_{i})-p(t)\bigg)f\left(\|p(t+h_{k}-h_{i})-p(t)\|^{2}\right)\nonumber\\
      &=m_{j}\Bigg(\bigg(p(t+h_{j}-h_{i})-p(t)\bigg)f\left(\|p(t+h_{j}-h_{i})-p(t)\|^{2}\right)\nonumber\\
      &-\bigg(p(t-(h_{j}-h_{k}))-p(t)\bigg)f\left(\|p(t-(h_{j}-h_{k}))-p(t)\|^{2}\right)\Bigg),
    \end{align}
    \begin{align}\label{Boom2}
      &(m_{i}-m_{k})\bigg(p(t+(h_{k}-h_{i}))-p(t)\bigg)f\left(\|p(t+(h_{k}-h_{i}))-p(t)\|^{2}\right)\nonumber\\
      &=m_{j}\Bigg(\bigg(p(t+(h_{j}-h_{i}))-p(t+(h_{k}-h_{i}))\bigg)f\left(\|p(t+(h_{j}-h_{i}))-p(t+(h_{k}-h_{i}))\|^{2}\right)\nonumber\\
      &-\bigg(p(t-(h_{j}-h_{k}))-p(t+(h_{k}-h_{i}))\bigg)f\left(\|p(t-(h_{j}-h_{k}))-p(t+(h_{k}-h_{i}))\|^{2}\right)\Bigg),
    \end{align}
    \begin{align}\label{Boom3}
      &(m_{i}-m_{k})\bigg(p(t+h_{k}-h_{i})-p(t)\bigg)\nonumber\\
      &=m_{j}\Bigg(\bigg(p(t+h_{j}-h_{i})-p(t)\bigg)-\bigg(p(t-(h_{j}-h_{k}))-p(t)\bigg)\Bigg),
    \end{align}
    and
    \begin{align}\label{Boom4}
      &(m_{i}-m_{k})\bigg(p(t+(h_{k}-h_{i}))-p(t)\bigg)\nonumber\\
      &=m_{j}\Bigg(\bigg(p(t+(h_{j}-h_{i}))-p(t+(h_{k}-h_{i}))\bigg)-\bigg(p(t-(h_{j}-h_{k}))-p(t+(h_{k}-h_{i}))\bigg)\Bigg).
    \end{align}
    If $p(t+h_{j}-h_{i})-p(t)$ and $p(t-(h_{j}-h_{k}))-p(t)$ are linearly independent, or pointing in the same direction, then for (\ref{Boom1}) and (\ref{Boom3}) to both be true, because $xf(x^{2})$ is a decreasing function, we need that these vectors are the same length, as otherwise $(m_{i}-m_{k})(p(t+h_{k}-h_{i})-p(t))$ has two different directions. By the same argument, working with (\ref{Boom2}) and (\ref{Boom4}) instead, we need that $p(t+(h_{j}-h_{i}))-p(t+(h_{k}-h_{i}))$ and $p(t-(h_{j}-h_{k}))-p(t+(h_{k}-h_{i}))$ have the same length if they are linearly independent, or pointing in the same direction. If for all $t$ at least one of these conditions is not met, then it is easy to see that using the continuity of $p$ all point masses lie on a straight line, which contradicts that they lie on a closed curve by \S 331 of \cite{W}. Thus we find that $p(t+(h_{j}-h_{i}))$ and $p(t-(h_{j}-h_{k}))$ lie on a circle of radius $\|p(t+(h_{k}-h_{i}))-p(t)\|$ around $p(t)$ and we find that $p(t+(h_{j}-h_{i}))$ and $p(t-(h_{j}-h_{k}))$ lie on a circle of radius $\|p(t+(h_{k}-h_{i}))-p(t)\|$ around $p(t+(h_{k}-h_{i}))$. This is only possible if the vectors on the right-hand sides of (\ref{Boom1}), (\ref{Boom2}), (\ref{Boom3}) and (\ref{Boom4}) cancel out against each other, which means that as $p(t+(h_{k}-h_{i}))-p(t)\neq 0$, we have that $m_{i}-m_{k}=0$ for all $i$, $k\in\{1,2,3\}$. This completes the proof.
  \section{Proof of Corollary~\ref{Main Corollary 2}}\label{Section proof of main corollary 2}
  If $q_{1}$,..., $q_{n}$ move along a curve like a choreography and $q_{n+1}=0$ with mass $m$, then, as in the proof of Corollary~\ref{Main Corollary 1}, by (\ref{Equations of motion}) we have for $k\neq n+1$ that
  \begin{align}\label{Corollary 2 Identity 1}
    \ddot{q}_{k}=\sum\limits_{\substack{j=1\\j\neq k}}^{n}m_{j}(q_{j}-q_{k})f\left(\|q_{j}-q_{k}\|^{2}\right)-m q_{k}f\left(\|q_{k}\|^{2}\right).
  \end{align}
  Using that $q_{j}(t)=p(t+h_{j})$ for $j<n+1$ and defining $s=t+h_{k}$, (\ref{Corollary 2 Identity 1}) can be rewritten as
  \begin{align}\label{Corollary 2 Identity 2}
     \ddot{p}(s)+m p(s)f\left(\|p(s)\|^{2}\right)=\sum\limits_{\substack{j=1\\j\neq k}}^{n}m_{j}\bigg(p(s+h_{j}-h_{k})-p(s)\bigg)f\left(\|p(s+h_{j}-h_{k})-p(s)\|^{2}\right).
  \end{align}
  Replacing $s$ with $-t$  in (\ref{Corollary 2 Identity 2}) and subtracting the resulting equation from (\ref{Corollary 2 Identity 2}), using that $p(-u)=\begin{pmatrix}1 & 0\\ 0 & -1\end{pmatrix}p(u)$ for any $u\in\mathbb{R}$ gives
  \begin{align}\label{Corollary 2 Identity 3}
     &\begin{pmatrix}1 & 0\\ 0 & -1\end{pmatrix}\bigg(\ddot{p}(t)+m p(t)f\left(\|p(t)\|^{2}\right)\bigg)\nonumber\\
     &=\begin{pmatrix}1 & 0\\ 0 & -1\end{pmatrix}\sum\limits_{\substack{j=1\\j\neq k}}^{n}m_{j}\bigg(p(t-(h_{j}-h_{k}))-p(t)\bigg)f\left(\|p(t-(h_{j}-h_{k}))-p(t)\|^{2}\right).
  \end{align}
  Multiplying (\ref{Corollary 2 Identity 3}) on both sides with $\begin{pmatrix}1 & 0\\ 0 & -1\end{pmatrix}^{-1}$ from the left and subtracting (\ref{Corollary 2 Identity 3}) from (\ref{Corollary 2 Identity 2}) (having replaced $s$ in (\ref{Corollary 2 Identity 2}) with $t$) gives (\ref{Equations of motion choreography result 1}), after which the proof of Lemma~\ref{Lemma 2} can be followed to the letter to obtain (\ref{Equations of motion choreography result 2}), (\ref{Equations of motion choreography result 3}) and (\ref{Equations of motion choreography result 4}), after which we can repeat the proof of Theorem~\ref{Main Theorem 2} to obtain our result. This completes the proof.
  \section{Proof of Theorem~\ref{Main Theorem 3}}\label{Section proof of main theorem 3}
  Let $q_{1}$,...,$q_{n}$ be an equally spaced choreographic solution of (\ref{EquationsOfMotion Curved}). Then by Lemma~\ref{Lemma 3} we have that
  \begin{align*}
       0&=\sum\limits_{j=1}^{n-1}\frac{\left(m_{j+k}-\frac{M}{n}\right)\Bigg(p(t+j)-\sigma\bigg(p(t+j)\odot p(t)\bigg)p(t)\Bigg)}{\Bigg(\sigma -\sigma\bigg(p(t+j)\odot p(t)\bigg)^{2}\Bigg)^{\frac{3}{2}}}
  \end{align*}
  for all $k\in\mathbb{Z}$.
  We again note, as in the proof of Theorem~\ref{Main Theorem 1}, that because $\widehat{e}_{1}$,..., $\widehat{e}_{n}$ span $\mathbb{C}^{n}$, there exist constants $a_{1},...,a_{n-1}$ such that
    \begin{align}\label{Span curved}
      \begin{pmatrix}
        m_{1+k}-\frac{M}{n}\\
        m_{2+k}-\frac{M}{n}\\
        \vdots\\
        m_{n+k}-\frac{M}{n}
      \end{pmatrix}=\sum\limits_{l=1}^{n-1}a_{l}\lambda_{l}^{k}\widehat{e}_{l}.
    \end{align}
    Again, note that by construction $\widehat{e}_{n}$ is orthogonal to $\left(m_{1+k}-\frac{M}{n},m_{2+k}-\frac{M}{n},...,m_{n+k}-\frac{M}{n}\right)^{T}$, so we may exclude $\widehat{e}_{n}$ in the linear combination in (\ref{Span curved}). Finally, analogous to the proof of Theorem~\ref{Main Theorem 1}, we find by (\ref{Span curved}) that proving that $m_{j+k}=\frac{M}{n}$ for all $j$, $k\in\mathbb{Z}$ is equivalent to proving that $a_{l}=0$ for all $l\in\{1,...,n-1\}$, which can be done by proving that
    \begin{align}\label{The L}
       0\neq\sum\limits_{j=1}^{n-1}\lambda_{l}^{j-1}\frac{\Bigg(p(t+j)-\sigma\bigg(p(t+j)\odot p(t)\bigg)p(t)\Bigg)}{\Bigg(\sigma -\sigma\bigg(p(t+j)\odot p(t)\bigg)^{2}\Bigg)^{\frac{3}{2}}}
  \end{align}
  for all $l\in\{1,...,n-1\}$. \\
  If $n=3$, we immediately get that the right-hand side of (\ref{The L}) cannot be real-valued and therefore cannot be zero.
  If $n=4$, we have that $\lambda_{l}=e^{\frac{2\pi i}{4}}$, $\lambda_{l}=e^{\frac{4\pi i}{4}}$, or $\lambda_{l}=e^{\frac{6\pi i}{4}}$. If $\lambda_{l}=e^{\frac{2\pi i}{4}}$, or $\lambda_{l}=e^{\frac{6\pi i}{4}}$, then the imaginary part of the right-hand side of (\ref{The L}) is \begin{align*}\pm \frac{p(t+2)-\sigma\bigg(p(t+2)\odot p(t)\bigg)p(t)}{\Bigg(\sigma -\sigma\bigg(p(t+2)\odot p(t)\bigg)^{2}\Bigg)^{\frac{3}{2}}},\end{align*} which cannot be zero. However, if $\lambda_{l}=e^{\frac{4\pi i}{4}}=-1$, we cannot use a similar argument to prove that the right-hand side of (\ref{The L}) is not zero, which means that $a_{2}$ is not necessarily zero, which translates into the possibility of $m_{1}=m_{3}$ and  $m_{2}=m_{4}$, but not all masses being equal.
  For $n=5$ we have that if (\ref{The L}) does not hold for a certain $l<5$, then
  \begin{align}\label{The NOT L1}
       0=\sum\limits_{j=1}^{4}\lambda_{l}^{j-1}\frac{p(t+j)-\sigma\bigg(p(t+j)\odot p(t)\bigg)p(t)}{\Bigg(\sigma -\sigma\bigg(p(t+j)\odot p(t)\bigg)^{2}\Bigg)^{\frac{3}{2}}}
  \end{align}
  and by Lemma~\ref{Lemma 3}
  \begin{align}\label{Help5}
       0&=\sum\limits_{j=1}^{4}\frac{\left(m_{j+k}-\frac{M}{5}\right)\Bigg(p(t+j)-\sigma\bigg(p(t+j)\odot p(t)\bigg)p(t)\Bigg)}{\Bigg(\sigma -\sigma\bigg(p(t+j)\odot p(t)\bigg)^{2}\Bigg)^{\frac{3}{2}}}.
  \end{align}
  Multiplying (\ref{The NOT L1}) on both sides with $m_{1+k}-\frac{M}{5}$ and subtracting the resulting equation from (\ref{Help5}) gives
  \begin{align}\label{Help5 1}
       0&=\sum\limits_{j=2}^{4}\frac{\left(m_{j+k}-\frac{M}{5}-\lambda_{l}^{j-1}\left(m_{1+k}-\frac{M}{5}\right)\right)\Bigg(p(t+j)-\sigma\bigg(p(t+j)\odot p(t)\bigg)p(t)\Bigg)}{\Bigg(\sigma -\sigma\bigg(p(t+j)\odot p(t)\bigg)^{2}\Bigg)^{\frac{3}{2}}}
  \end{align}
  and taking complex conjugates on both sides of (\ref{Help5 1}) gives
  \begin{align}\label{Help5 2}
       0&=\sum\limits_{j=2}^{4}\frac{\left(m_{j+k}-\frac{M}{5}-\lambda_{5-l}^{j-1}\left(m_{1+k}-\frac{M}{5}\right)\right)\Bigg(p(t+j)-\sigma\bigg(p(t+j)\odot p(t)\bigg)p(t)\Bigg)}{\Bigg(\sigma -\sigma\bigg(p(t+j)\odot p(t)\bigg)^{2}\Bigg)^{\frac{3}{2}}}.
  \end{align}
  Note that
  \begin{align}\label{MassVector1}
    \begin{pmatrix}
      m_{2+k}-\frac{M}{5}-\lambda_{l}\left(m_{1+k}-\frac{M}{5}\right)\\
      m_{3+k}-\frac{M}{5}-\lambda_{l}^{2}\left(m_{1+k}-\frac{M}{5}\right)\\
      m_{4+k}-\frac{M}{5}-\lambda_{l}^{3}\left(m_{1+k}-\frac{M}{5}\right)
    \end{pmatrix}
    \end{align}
    and
    \begin{align}
    \begin{pmatrix}\label{MassVector2}
      m_{2+k}-\frac{M}{5}-\lambda_{5-l}\left(m_{1+k}-\frac{M}{5}\right)\\
      m_{3+k}-\frac{M}{5}-\lambda_{5-l}^{2}\left(m_{1+k}-\frac{M}{5}\right)\\
      m_{4+k}-\frac{M}{5}-\lambda_{5-l}^{3}\left(m_{1+k}-\frac{M}{5}\right)
    \end{pmatrix}
  \end{align}
  are linearly independent if and only if $m_{1+k}-\frac{M}{5}\neq 0$. Additionally, note that using the Euclidean inner product, the vectors
  \begin{align*}
    &\frac{p(t+2)-\sigma\bigg(p(t+2)\odot p(t)\bigg)p(t)}{\Bigg(\sigma -\sigma\bigg(p(t+2)\odot p(t)\bigg)^{2}\Bigg)^{\frac{3}{2}}},\textrm{ }\frac{p(t+3)-\sigma\bigg(p(t+3)\odot p(t)\bigg)p(t)}{\Bigg(\sigma -\sigma\bigg(p(t+3)\odot p(t)\bigg)^{2}\Bigg)^{\frac{3}{2}}}\\
    &\textrm{and }\frac{p(t+4)-\sigma\bigg(p(t+4)\odot p(t)\bigg)p(t)}{\Bigg(\sigma -\sigma\bigg(p(t+4)\odot p(t)\bigg)^{2}\Bigg)^{\frac{3}{2}}}
  \end{align*}
  in (\ref{Help5 1}) and (\ref{Help5 2}) are orthogonal to
  \begin{align*}
    v(t):=\begin{pmatrix}
      1 & 0 & 0 \\
      0 & 1 & 0 \\
      0 & 0 & \sigma
    \end{pmatrix}p(t)
  \end{align*}
  and span a two-dimensional space if and only if $v(t)$, $p(t+2)$, $p(t+3)$ and $p(t+4)$ do not lie in the same plane. For $\sigma=-1$, by the geometry of the hyperboloid, we directly see that $v(t)$, $p(t+2)$, $p(t+3)$ and $p(t+4)$ do not lie in the same plane. For $\sigma=1$ we have that $v(t)=p(t)$ and $p(t)$, $p(t+2)$, $p(t+3)$ and $p(t+4)$ may lie in the same plane if that plane goes through the center of the sphere, in which case the point masses $p(t)$, $p(t+2)$, $p(t+3)$ and $p(t+4)$ lie on a great circle. Because $p(t)$, $p(t+2)$, $p(t+3)$ and $p(t+4)$ then lie on a great circle and lie in the same plane, we can write $p(t)=a(t)p(t+2)+b(t)p(t+3)$, where $a$ and $b$ are scalar functions on $\mathbb{R}$, which means that $p(t+1)=a(t+1)p(t+3)+b(t+1)p(t+4)$, so because $p(t+3)$ and $p(t+4)$ lie in the same plane, $p(t+1)$ lies in that plane and all the point masses lie on a great circle. So that means that unless the point masses lie on a great circle, we have that
  \begin{align*}
    &\frac{p(t+2)-\sigma\bigg(p(t+2)\odot p(t)\bigg)p(t)}{\Bigg(\sigma -\sigma\bigg(p(t+2)\odot p(t)\bigg)^{2}\Bigg)^{\frac{3}{2}}},\textrm{ }\frac{p(t+3)-\sigma\bigg(p(t+3)\odot p(t)\bigg)p(t)}{\Bigg(\sigma -\sigma\bigg(p(t+3)\odot p(t)\bigg)^{2}\Bigg)^{\frac{3}{2}}}\\
    &\textrm{and }\frac{p(t+4)-\sigma\bigg(p(t+4)\odot p(t)\bigg)p(t)}{\Bigg(\sigma -\sigma\bigg(p(t+4)\odot p(t)\bigg)^{2}\Bigg)^{\frac{3}{2}}}
  \end{align*}
  are three vectors that span a two-dimensional space (if the point masses do not lie on a great circle), which means that if (\ref{MassVector1}) and (\ref{MassVector2}) are linearly independent, we have a contradiction by (\ref{Help5 1}) and (\ref{Help5 2}). So this means that $m_{1+k}-\frac{M}{5}=0$, which proves that all masses are equal if $n=5$.
  This then finally proves that if $\sigma=-1$, if $n<6$ and $n\neq 4$, all masses are equal and that if $n=4$ the odd labeled masses are equal and the even labeled masses are equal and that the same holds true if $\sigma=1$, unless the point masses all move along a great circle.
  \section{Proof of Corollary~\ref{Main Corollary 3}}\label{Section proof of main corollary 3}
  Let $q_{1},...,q_{n+1}$ be a solution of (\ref{EquationsOfMotion Curved}), where $q_{1},...,q_{n}$ is an equally spaced choreography and $q_{n+1}=(0,0,1)^{T}$ with mass $m$. Let $\widehat{e}=(0,0,1)^{T}$. Then by (\ref{EquationsOfMotion Curved}) we have for $k\in\{1,...,n\}$ that
  \begin{align*}
   &\ddot{p}(t+k)+\sigma\bigg(\dot{p}(t+k)\odot\dot{p}(t+k)\bigg)p(t+k)\\
   &=\sum\limits_{j=1,\textrm{ }j\neq k}^{n}\frac{m_{j}\Bigg(p(t+j)-\sigma\bigg(p(t+k)\odot p(t+j)\bigg)p(t+k)\Bigg)}{\Bigg(\sigma -\sigma\bigg(p(t+k)\odot p(t+j)\bigg)^{2}\Bigg)^{\frac{3}{2}}}\\
   &+\frac{m\Bigg(\widehat{e}-\sigma\bigg(p(t+k)\odot \widehat{e}\bigg)p(t+k)\Bigg)}{\Bigg(\sigma -\sigma\bigg(p(t+k)\odot \widehat{e}\bigg)^{2}\Bigg)^{\frac{3}{2}}},
  \end{align*}
  which, writing $s=t+k$, can be rewritten as
  \begin{align}\label{Final}
   &\ddot{p}(s)+\sigma\bigg(\dot{p}(s)\odot\dot{p}(s)\bigg)p(s)\nonumber\\
   &=\sum\limits_{j=1}^{n-1}\frac{m_{j+k}\Bigg(p(s+j)-\sigma\bigg(p(s)\odot p(s+j)\bigg)p(s)\Bigg)}{\Bigg(\sigma -\sigma\bigg(p(s)\odot p(s+j)\bigg)^{2}\Bigg)^{\frac{3}{2}}}+\frac{m\Bigg(\widehat{e}-\sigma\bigg(p(s)\odot \widehat{e}\bigg)p(s)\Bigg)}{\Bigg(\sigma -\sigma\bigg(p(s)\odot \widehat{e}\bigg)^{2}\Bigg)^{\frac{3}{2}}}.
  \end{align}
  Summing both sides of (\ref{Final}) with respect to $k$ from $1$ to $n$, dividing both sides of the resulting equation by $n$ and subtracting that equation from (\ref{Final}) then gives (\ref{CurvedHelp}), which is the result on which the proof of Theorem~\ref{Main Theorem 3} was based.
   So if $\sigma=-1$, if $n<6$ and $n\neq 4$, then the masses $m_{j}$, $j\in\{1,...,n\}$ are equal and if $n=4$ the odd labeled masses $m_{j}$ are equal and the even labeled masses $m_{j}$ are equal, $j\in\{1,...,n\}$ and the same holds true for $\sigma=1$, unless the point masses all move along a great circle.
   \section{Acknowledgements}
   The author would like to thank Dr. Pascal Grange and Dr. Yinna Ye of Xi'an Jiaotong-Liverpool University for their useful feedback during the final stage of the writing of this paper.

\end{document}